\documentclass{birkmult}

\usepackage{latexsym,amsmath,amscd,amssymb,graphics,xypic}
\usepackage{enumerate}

\usepackage{graphicx,lscape}

\usepackage[colorlinks]{hyperref}
\usepackage{url}

\begin{document}

\newtheorem{theorem}{Theorem}[section]
\newtheorem{definition}[theorem]{Definition}
\newtheorem{lemma}[theorem]{Lemma}
\newtheorem{remark}[theorem]{Remark}
\newtheorem{proposition}[theorem]{Proposition}
\newtheorem{corollary}[theorem]{Corollary}
\newtheorem{example}[theorem]{Example}

\numberwithin{equation}{section}
\newcommand{\ep}{\varepsilon}
\newcommand{\R}{{\mathbb  R}}
\newcommand\C{{\mathbb  C}}
\newcommand\Q{{\mathbb Q}}
\newcommand\Z{{\mathbb Z}}
\newcommand{\N}{{\mathbb N}}

\title[A unified formulation of completely integrable systems]{On a unified formulation of completely integrable systems}

\author{R\u{a}zvan M. Tudoran}

\address{The West University of Timi\c soara,\\
Faculty of Mathematics and Computer Science, \\
Department of Mathematics,\\
Blvd. Vasile P\^ arvan, No. 4,\\
300223 - Timi\c soara,\\
Romania.}

\email{tudoran@math.uvt.ro}

\subjclass{37J35; 37K10; 70H05; 70H06.}

\keywords{integrable systems; Hamiltonian dynamics; linear normal forms.}

\date{June 23, 2011}
\dedicatory{}
%%% ----------------------------------------------------------------------

\begin{abstract}
The purpose of this article is to show that a $\mathcal{C}^1$ differential system on $\R^n$ which admits a set of $n-1$ independent $\mathcal{C}^2$ conservation laws defined on an open subset $\Omega\subseteq \R^n$, is essentially $\mathcal{C}^1$ equivalent on an open and dense subset of $\Omega$, with the linear differential system $u^\prime_1=u_1, \ u^\prime_2=u_2, \dots, \ u^\prime_n=u_n$. The main results are illustrated in the case of two concrete dynamical systems, namely the three dimensional Lotka-Volterra system, and respectively the Euler equations from the free rigid body dynamics.  
\end{abstract}

\maketitle

\section{Introduction}
\label{section:one}

Recently, in \cite{llibre} it is proved that an integrable $\mathcal{C}^1$ planar differential system is roughly speaking $\mathcal{C}^1$equivalent to the linear differential system $u^\prime_1=u_1, \ u^\prime_2=u_2$. 

The purpose of this article is to generalize this result in the $n$ dimensional case for a $\mathcal{C}^1$ differential system that admits a set of $n-1$ independent conservation laws. In the second section we show that such a system can always be realized as a Hamilton-Poisson dynamical system on a full measure open subset of $\R^n$ with respect to a rank 2 Poisson structure. In the third section a new time transformation will be explicitly constructed in order to bring the system to a linear differential system of the type $u^\prime_1=u_1, \ u^\prime_2=u_2, \dots, \ u^\prime_n=u_n$. In the last section we illustrate the main results in the case of two concrete dynamical systems, namely the three dimensional Lotka-Volterra system, and respectively the Euler equations from the free rigid body dynamics. 

For details on Poisson geometry and Hamiltonian dynamics, see, e.g. \cite{abraham}, \cite{arnoldcarte}, \cite{marsdenratiu}, \cite{holm1}, \cite{holm2}, \cite{holm3}, \cite{ratiurazvan}.

\section{Hamiltonian divergence free vector fields naturally associated to integrable systems}

In this section we give a method to construct a Hamilton-Poisson divergence free vector field, naturally associated with a given  Hamilton-Poisson realization of a $n$ dimensional differential system admitting $n-1$ independent conservation laws.

First step in this approach is to construct a Hamilton-Poisson realization of a given $n$ dimensional differential system admitting $n-1$ independent integrals of motion.

Let us consider a $\mathcal{C}^1$ differential system on $\R^n$:
\begin{equation}\label{sys}
\left\{\begin{array}{l}
\dot x_{1}=X_1(x_1,\dots,x_n)\\
\dot x_{2}=X_2(x_1,\dots,x_n)\\
\cdots\\
\dot x_{n}=X_n(x_1,\dots,x_n),\\
\end{array}\right.
\end{equation}
where $X_1,X_2,\dots,X_n\in\mathcal{C}^{1}(\R^n,\R)$ are arbitrary real functions. Suppose that $C_1,\dots,C_{n-2},C_{n-1}:\Omega\subseteq \R^n\rightarrow \R$ are $n-1$ independent $\mathcal{C}^2$ integrals of motion of \eqref{sys} defined on a nonempty open subset $\Omega\subseteq \R^n$.

Since $C_1,\dots,C_{n-2},C_{n-1}:\Omega\subseteq \R^n\rightarrow \R$ are integrals of motion of the vector field $X=X_1\partial_{x_1}+\dots+X_n \partial_{x_n}\in\mathfrak{X}(\R^n)$, we obtain that for each $i\in\{1,\dots,n-1\}$
$$
\langle\nabla C_i(x),X(x)\rangle=\sum_{j=1}^{n}\partial_{x_j}C_i \cdot\dot{x}_{j}=0,
$$
for every $x=(x_1,\dots,x_n)\in\Omega$, where $\langle \cdot,\cdot\rangle$ stand for the canonical inner product on $\R^n$, and respectively $\nabla$ stand for the gradient with respect to $\langle \cdot,\cdot\rangle$.

Hence, by a standard multilinear algebra argument, the $\mathcal{C}^{1}$ vector field $X$ is given as the $\mathcal{C}^{1}$ vector field $\star (\nabla C_1\wedge\dots\wedge \nabla C_{n-1})$ multiplied by a $\mathcal{C}^{1}$ real function (rescaling function), where $\star$ stand for the Hodge star operator for multivector fields (see for details e.g. \cite{darling}). It may happen that the domain of definition for the rescaling function to be a proper subset of $\Omega$. In the following we will consider the generic case when the rescaling function is defined on an open and dense subset of $\Omega$. In order to simplify the notations, we will also denote this set by $\Omega$.  

Consequently, the vector field $X$ can be realized on the open set $\Omega\subseteq \R^n$ as the Hamilton-Poisson vector field $X_H \in\mathfrak{X}(\Omega)$ with respect to the Hamiltonian function $H:=C_{n-1}$ and respectively the Poisson bracket of class $\mathcal{C}^{1}$ defined by:
$$
\{f,g\}_{\nu;C_1,\dots,C_{n-2}}dx_1\wedge\dots\wedge dx_n=\nu dC_1\wedge\dots dC_{n-2}\wedge df\wedge dg,
$$
where $\nu\in\mathcal{C}^{1}(\Omega,\R)$ is a given real function (rescaling). For $\nu\equiv 1$, the associated Poisson bracket in the smooth category, it is exactly the Flaschka-Ra\c tiu bracket. For similar Hamilton-Poisson formulations of completely integrable systems see also \cite{g1}, \cite{g2}.

In coordinates, the bracket $\{f,g\}_{\nu;C_1,\dots,C_{n-2}}$ is given by:
$$
\{f,g\}_{\nu;C_1,\dots,C_{n-2}}=\nu \cdot \dfrac{\partial(C_1,\dots,C_{n-2},f,g)}{\partial(x_1,\dots,x_n)}.
$$

Note that $\{C_1,\dots,C_{n-2}\}$ is a complete set of Casimirs for the Poisson bracket $\{\cdot,\cdot\}_{\nu;C_1,\dots,C_{n-2}}$.

Recall that the Hamiltonian vector field $X_H\in\mathfrak{X}(\Omega)$ is acting on an arbitrary real function $f\in \mathcal{C}^{k}(\Omega,\R)$, $k\geq 2$ as: 
$$X_{H}(f)=\{f,H\}_{\nu;C_1,\dots,C_{n-2}}\in \mathcal{C}^{1}(\Omega,\R).$$

Hence, the differential system \eqref{sys} can be locally written in $\Omega$ as a Hamilton-Poisson dynamical system of the type:
$$\left\{\begin{array}{l}
\dot x_{1}=\{x_1,H\}_{\nu;C_1,\dots,C_{n-2}}\\
\dot x_{2}=\{x_2,H\}_{\nu;C_1,\dots,C_{n-2}}\\
\cdots \\
\dot x_{n}=\{x_n,H\}_{\nu;C_1,\dots,C_{n-2}},\\
\end{array}\right.
$$
or equivalently
\begin{equation}\label{sysham}
\left\{\begin{array}{l}
\dot x_{1}=\nu \cdot \dfrac{\partial(C_1,\dots,C_{n-2},x_1,H)}{\partial(x_1,\dots,x_n)}\\
\dot x_{2}=\nu \cdot \dfrac{\partial(C_1,\dots,C_{n-2},x_2,H)}{\partial(x_1,\dots,x_n)}\\
\cdots \\
\dot x_{n}=\nu \cdot \dfrac{\partial(C_1,\dots,C_{n-2},x_n,H)}{\partial(x_1,\dots,x_n)}.\\
\end{array}\right.
\end{equation}

Consequently, the components of the vector field $X=X_1 \partial_{x_1}+\dots+X_n \partial_{x_n}$ which generates the differential system \eqref{sys}, are given in $\Omega$ as follows:
$$X_i=\nu \cdot \dfrac{\partial(C_1,\dots,C_{n-2},x_i,H)}{\partial(x_1,\dots,x_n)},$$
for $i\in\{1,\dots,n\}$.

Next result gives a method to construct a divergence free vector field out of the vector field $X$. The divergence operator we will use in this approach is the divergence associated with the standard Lebesgue measure on $\R^n$, namely $\mathcal{L}_{X}dx_1\wedge\dots\wedge dx_n=\operatorname{div}{X}dx_1\wedge\dots\wedge dx_n$, where $\mathcal{L}_{X}$ stand for the Lie derivative along the vector field $X$.

\begin{theorem}\label{divfree}
The vector field $\widetilde {X}:=\dfrac{1}{\nu}\cdot X$ is a divergence free vector field on $\Omega \setminus \mathcal{Z}(\nu)$, where $\mathcal{Z}(\nu)=\{(x_1,\dots,x_n)\in\Omega \mid \nu(x_1,\dots,x_n)=0\}$.
\end{theorem}
\begin{proof}
Note that the components of the vector field $\widetilde {X}=\widetilde{X}_1 \partial_{x_1}+\dots+\widetilde{X}_n \partial_{x_n}$ are given by:
$$
\widetilde{X}_i=\dfrac{\partial(C_1,\dots,C_{n-2},x_i,H)}{\partial(x_1,\dots,x_n)},
$$
for $i\in\{1,\dots,n\}$. By definition, the vector field $\widetilde{X}$ is a Hamilton-Poisson vector field with respect to the Flaschka-Ra\c tiu bracket, and having the same Hamiltonian $H$ as the vector field $X$.

Hence, the divergence of $\widetilde{X}$ is given by:
\begin{align*}
\operatorname{div}(\widetilde{X})&=\sum_{i=1}^{n}\partial_{x_i} \widetilde{X}_i=\sum_{i=1}^{n}\partial_{x_i}\dfrac{\partial(C_1,\dots,C_{n-2},x_i,H)}{\partial(x_1,\dots,x_n)}\\
&=\sum_{i=1}^{n}(-1)^{i+n-1}\partial_{x_i}\dfrac{\partial(C_1,\dots,C_{n-2},H)}{\partial(x_1,\dots,\widehat{x}_{i},\dots,x_n)},
\end{align*}
where the notation "$\widehat{x}_i$" means that "$x_i$" is omitted. 

Let us now analyze the general term in the above sum. By using the derivative of a determinant we obtain the following:
\begin{align*}
\partial_{x_i}\dfrac{\partial(C_1,\dots,C_{n-2},H)}{\partial(x_1,\dots,\widehat{x}_{i},\dots,x_n)}&=\dfrac{\partial(\partial_{x_i}C_1,\dots,C_{n-2},H)}{\partial(x_1,\dots,\widehat{x}_{i},\dots,x_n)}+\dots+\dfrac{\partial(C_1,\dots,\partial_{x_i}C_{n-2},H)}{\partial(x_1,\dots,\widehat{x}_{i},\dots,x_n)}\\
&+\dfrac{\partial(C_1,\dots,C_{n-2},\partial_{x_i}H)}{\partial(x_1,\dots,\widehat{x}_{i},\dots,x_n)}.
\end{align*}
Hence,
\begin{align*}
&\operatorname{div}(\widetilde{X})=\sum_{i=1}^{n}(-1)^{i+n-1}( \dfrac{\partial(\partial_{x_i}C_1,\dots,C_{n-2},H)}{\partial(x_1,\dots,\widehat{x}_{i},\dots,x_n)}+\dots+\dfrac{\partial(C_1,\dots,\partial_{x_i}C_{n-2},H)}{\partial(x_1,\dots,\widehat{x}_{i},\dots,x_n)}\\
&+\dfrac{\partial(C_1,\dots,C_{n-2},\partial_{x_i}H)}{\partial(x_1,\dots,\widehat{x}_{i},\dots,x_n)} )\\
&=\sum_{i=1}^{n}(-1)^{i+n-1}\dfrac{\partial(\partial_{x_i}C_1,\dots,C_{n-2},H)}{\partial(x_1,\dots,\widehat{x}_{i},\dots,x_n)}+\dots+
\sum_{i=1}^{n}(-1)^{i+n-1}\dfrac{\partial(C_1,\dots,\partial_{x_i}C_{n-2},H)}{\partial(x_1,\dots,\widehat{x}_{i},\dots,x_n)}\\
&+\sum_{i=1}^{n}(-1)^{i+n-1}\dfrac{\partial(C_1,\dots,C_{n-2},\partial_{x_i}H)}{\partial(x_1,\dots,\widehat{x}_{i},\dots,x_n)}.
\end{align*}
Next we prove that each of the above sums vanishes. In order to do that, it is enough to show that the general sum $S_{k}$ vanishes, where
$$
S_{k}:=\sum_{i=1}^{n}(-1)^{i+n-1}\dfrac{\partial(C_1,\dots,C_{k-1},\partial_{x_i}C_{k},C_{k+1},\dots,C_{n-2},H)}{\partial(x_1,\dots,\widehat{x}_{i},\dots,x_n)}.
$$ 
Indeed, we obtain that:
\begin{align*}
&S_{k}=\sum_{i=1}^{n}(-1)^{i+n-1}\sum_{j=1}^{i-1}(-1)^{j+k}\partial^{2}_{x_j x_i}C_{k} \cdot \dfrac{\partial(C_1,\dots,\widehat{C}_k,\dots,C_{n-2},H)}{\partial(x_1,\dots,\widehat{x}_j,\dots,\widehat{x}_i,\dots,x_n)}\\
&+\sum_{i=1}^{n}(-1)^{i+n-1}\sum_{j=i+1}^{n}(-1)^{j+k-1}\partial^{2}_{x_j x_i}C_{k} \cdot \dfrac{\partial(C_1,\dots,\widehat{C}_k,\dots,C_{n-2},H)}{\partial(x_1,\dots,\widehat{x}_i,\dots,\widehat{x}_j,\dots,x_n)}\\
&=\sum_{i=1}^{n}(-1)^{i+n-1}\sum_{j=1}^{i-1}(-1)^{j+k}\partial^{2}_{x_j x_i}C_{k} \cdot \dfrac{\partial(C_1,\dots,\widehat{C}_k,\dots,C_{n-2},H)}{\partial(x_1,\dots,\widehat{x}_j,\dots,\widehat{x}_i,\dots,x_n)}\\
&+\sum_{j=1}^{n}(-1)^{j+n-1}\sum_{i=j+1}^{n}(-1)^{i+k-1}\partial^{2}_{x_i x_j}C_{k} \cdot \dfrac{\partial(C_1,\dots,\widehat{C}_k,\dots,C_{n-2},H)}{\partial(x_1,\dots,\widehat{x}_j,\dots,\widehat{x}_i,\dots,x_n)}\\
&=(-1)^{n+k-1}\sum_{1\leq j<i\leq n}[(-1)^{i+j}\partial^{2}_{x_j x_i}C_{k} +(-1)^{i+j-1}\partial^{2}_{x_i x_j}C_{k}]\cdot\dfrac{\partial(C_1,\dots,\widehat{C}_k,\dots,C_{n-2},H)}{\partial(x_1,\dots,\widehat{x}_j,\dots,\widehat{x}_i,\dots,x_n)}\\
&=0.
\end{align*}
\end{proof} 
\begin{remark}
In the symplectic case, each Hamiltonian vector field is divergence free with respect to the divergence operator defined by the Liouville volume form.
\end{remark}

\section{A unified linear formulation of integrable systems}

In this section we give a unified linear formulation for the $n$ dimensional differential systems admitting a set of $n-1$ independent conservation laws. The construction of this linear formulation makes use explicitly of the Hamilton-Poisson realization \eqref{sysham} of a differential system of type \eqref{sys}.

Let us now state the main result of this article. The notations and respectively the hypothesis are supposed to be the same as in the previous section.

\begin{theorem}\label{mainresult}
Let \eqref{sys} be a $\mathcal{C}^{1}$ differential system having a set of $n-1$ independent $\mathcal{C}^2$ conservation laws defined on an open subset $\Omega\subseteq \R^n$. Assume that there exists a $\mathcal{C}^1$ rescaling function $\nu$, nonzero on an open and dense subset $\Omega_0$ of $\Omega$, such that the system \eqref{sys} admits a Hamilton-Poisson realization of the type \eqref{sysham} and the Lebesgue measure of the set $$\mathcal{O}:=\left\{ x=(x_1,\dots,x_n)\in \Omega_{0} \mid \operatorname{div}(X)(x)\cdot\dfrac{\partial(1/\nu,C_{1},\dots,C_{n-2},H)}{\partial(x_1,\dots,x_n)}(x)=0\right\}$$ in $\Omega_{0}$ is zero. 

Then, the change of variables $((x_1,\dots,x_n),t)\mapsto ((u_1,\dots,u_n),s)$ given by
\begin{equation*}
\left\{\begin{array}{l}
u_1 =1/\nu(x_1,\dots ,x_n)\\
u_2 =C_1(x_1,\dots ,x_n)/\nu(x_1,\dots ,x_n)\\
\cdots \\
u_{n-1} =C_{n-2}(x_1,\dots ,x_n)/\nu(x_1,\dots ,x_n)\\
u_n =H(x_1,\dots ,x_n)/\nu(x_1,\dots ,x_n)\\
ds=-\operatorname{div}(X)dt,
\end{array}\right.
\end{equation*}
in the open and dense subset $\Omega_{00}:=\Omega_{0}\setminus \mathcal{O}$ of $\Omega_0$, transforms the system \eqref{sysham} restricted to $\Omega_{00}$ into the linear differential system $u^\prime_1=u_1, \ u^\prime_2=u_2, \dots, \ u^\prime_n=u_n$, where "prime" stand for the derivative with respect to the new time "$s$". 
\end{theorem}
\begin{proof}
Let us start by recalling that the change of variables $(x_1,\dots,x_n)\longmapsto (u_1,\dots,u_n)$ is of class $\mathcal{C}^1$ on $\Omega_{00}$. More exactly, the change of variables is defined on a wider set, namely $\Omega_{0}\setminus \mathcal{E}$, where
\begin{align*}
\mathcal{E}:&=\left\{ x=(x_1,\dots,x_n)\in \Omega_{0} \mid \dfrac{\partial(u_1,\dots,u_n)}{\partial(x_1,\dots,x_n)}(x)=0\right\}\\
&=\left\{ x=(x_1,\dots,x_n)\in \Omega_{0} \mid (1/\nu)^{n-1}(x)\cdot\dfrac{\partial(1/\nu,C_{1},\dots,C_{n-2},H)}{\partial(x_1,\dots,x_n)}(x)=0\right\}\\
&=\left\{ x=(x_1,\dots,x_n)\in \Omega_{0} \mid \dfrac{\partial(1/\nu,C_{1},\dots,C_{n-2},H)}{\partial(x_1,\dots,x_n)}(x)=0\right\}.
\end{align*}
In order to determine the transformed differential system \eqref{sys} (generated by $X$) through this change of variables, we need the following result. Recall by Theorem \eqref{divfree} that the vector filed $1/\nu \cdot X$ is a divergence free vector field on $\Omega_{00}$. Since
\begin{align*}
\operatorname{div}(1/\nu \cdot X)=\langle\nabla(1/\nu),X\rangle + 1/\nu \cdot\operatorname{div}(X),
\end{align*}
and $\operatorname{div}(1/\nu \cdot X)=0$, we obtain that:
\begin{equation}\label{util}
\langle\nabla(1/\nu),X\rangle =-1/\nu \cdot\operatorname{div}(X).
\end{equation}
After the change of variables in $\Omega_{00}$, the system \eqref{sys} becomes:
\begin{equation}\label{systrans}
\left\{\begin{array}{l}
\dfrac{du_{1}}{dt}=\dfrac{d}{dt}(1/\nu)=\langle\nabla(1/\nu),X\rangle=- \operatorname{div}(X)\cdot 1/\nu =- \operatorname{div}(X)\cdot u_1\\
\dfrac{du_{2}}{dt}=\dfrac{d}{dt}(1/\nu \cdot C_1)=C_{1}\cdot \langle\nabla(1/\nu),X \rangle + 1/\nu \cdot \langle \nabla C_1,X\rangle=- \operatorname{div}(X)\cdot u_2\\
\cdots\\
\dfrac{du_{n-1}}{dt}=\dfrac{d}{dt}(1/\nu \cdot C_{n-2})=C_{n-2}\cdot \langle\nabla(1/\nu),X \rangle + 1/\nu \cdot \langle \nabla C_{n-2},X\rangle=- \operatorname{div}(X)\cdot u_{n-1}\\
\dfrac{du_{n}}{dt}=\dfrac{d}{dt}(1/\nu \cdot H)=H \cdot \langle\nabla(1/\nu),X \rangle + 1/\nu \cdot \langle \nabla H,X\rangle=- \operatorname{div}(X)\cdot u_n,\\
\end{array}\right.
\end{equation}
where we used the relation \eqref{util} and the fact that $C_1,\dots,C_{n-2},H$ are conservation laws for the vector field $X$, and consequently 
$$
\langle \nabla C_{1},X\rangle=\dots=\langle \nabla C_{n-2},X\rangle=\langle \nabla H,X\rangle=0.
$$
Now using the new time transformation $ds=-\operatorname{div}(X)dt$ on $\Omega_{00}$, the system \eqref{systrans} becomes:
\begin{equation*}
\left\{\begin{array}{l}
u^\prime_1=u_1\\
u^\prime_2=u_2\\
\cdots\\
u^\prime_n=u_n,\\
\end{array}\right.
\end{equation*}
where $u^\prime_i=\dfrac{du_{i}}{ds}$, for $i\in\{1,\dots,n\}$.
\end{proof}

\begin{remark}\label{okrem}
If the rescaling function $\nu=:\nu_{cst.}$ is a constant function, then the Lebesgue measure of the set $\mathcal{O}=\Omega_0=\Omega$ is nonzero in $\Omega_0$, and hence the assumptions of the Theorem \eqref{mainresult} do not hold. In this case, we search for a new $\mathcal{C}^1$ rescaling function $\mu$ defined on an open and dense subset $\Omega_0$ of $\Omega$, such that the vector field $\mu \cdot X$ satisfies the assumptions of the Theorem \eqref{mainresult}. The function $\mu$ satisfies:
$$
\operatorname{div}(\mu \cdot X)=\langle\nabla\mu,X\rangle + \mu \cdot\operatorname{div}(X)=\langle\nabla\mu,X\rangle,
$$
since in the case of a constant function $\nu=\nu_{cst.}$ we obtain from Theorem \eqref{divfree} that
$$
0=\operatorname{div}(1/\nu_{cst.} \cdot X)=(1/\nu_{cst.}) \cdot\operatorname{div}(X),
$$
and hence $\operatorname{div}(X)=0$. 

Consequently, we have to search for a rescaling function $\mu$ such that $$\operatorname{div}(\mu \cdot X)=\langle\nabla\mu,X\rangle$$ it is not identically zero in $\Omega_0$.

The second condition that $\mu$ has to satisfy is that the function $$\dfrac{\partial(1/\mu,C_{1},\dots,C_{n-2},H)}{\partial(x_1,\dots,x_n)}$$ it is not identically zero in $\Omega_0$.

The transformation between the differential system \eqref{sys} and the differential system generated by the vector field $\mu \cdot X$ is done by using the new time transformation $dt=\mu(x)dt^{\prime}$. More exactly, the differential system \eqref{sys} generated by $X$, namely:
\begin{equation*}
\dfrac{dx}{dt}=X(x),
\end{equation*}
is transformed via the new time transformation $dt=\mu(x)dt^{\prime}$ into the system:
\begin{equation}\label{sysy}
\dfrac{dx}{dt^{\prime}}=\mu(x)\cdot X(x).\\
\end{equation}
The differential system \eqref{sysy} can also be realized as a Hamilton-Poisson dynamical system with respect to the Poisson bracket $\{\cdot,\cdot\}_{\mu\cdot\nu_{cst.};C_1,\dots,C_{n-2}}$ and respectively the same Hamiltonian $H$ as for the Hamilton-Poisson realization \eqref{sysham} of the vector field $X$.
\end{remark}

\section{Examples}

In this section we will apply the main result of this article in the case of two concrete dynamical systems, namely a 3D Lotka-Volterra system and respectively Euler's equations of the free rigid body dynamics.

Let us start with the Lotka-Volterra system. The 3D Lotka-Volterra system we consider (see e.g. \cite{lvs}), is described by the following differential system:
\begin{equation}\label{LV}
\left\{\begin{array}{l}
\dfrac{dx_{1}}{dt}=x_1 (x_2 +x_3)\\
\dfrac{dx_{2}}{dt}=x_2 (-x_1 +x_3)\\
\dfrac{dx_{3}}{dt}=x_ 3(-x_1 -x_2).\\
\end{array}\right.
\end{equation}
If one denote: $$X=[x_1 (x_2 +x_3)]\partial_{x_1}+[x_2 (-x_1 +x_3)]\partial_{x_2}+[x_3 (-x_1 -x_2)]\partial_{x_3},$$ then $\operatorname{div}(X)=2(x_3-x_1)$.

The system \eqref{LV} admits a Hamilton-Poisson realization of the type \eqref{sysham}, where:
\begin{align*}
&\nu(x_1,x_2,x_3)=-\dfrac{x_1^{2} x_3^{2}}{x_1 +x_2 +x_3},\\
&C(x_1,x_2,x_3)=\dfrac{x_2(x_1 +x_2 +x_3)}{x_1 x_3},\\
&H(x_1,x_2,x_3)=x_1 +x_2 +x_3.
\end{align*}
The sets introduced in Theorem \eqref{mainresult} in the case of the Lotka-Volterra system \eqref{LV} are given by:
\begin{align*}
&\Omega=\{(x_1,x_2,x_3)\in\R^3 \mid x_1 x_3 \neq 0\},\\
&\Omega_{0}=\{(x_1,x_2,x_3)\in\R^3 \mid x_1 x_3 \neq 0;\ x_1+x_2 +x_3 \neq 0\},\\
&\mathcal{O}=\{(x_1,x_2,x_3)\in\R^3 \mid x_1 =x_3;\ x_1 x_3 \neq 0;\ x_1 +x_2 + x_3 \neq 0\},\\
&\Omega_{00}=\{(x_1,x_2,x_3)\in\R^3 \mid x_1\neq x_3;\ x_1 x_3 \neq 0; x_1+x_2 +x_3 \neq 0\}.
\end{align*}
Then by Theorem \eqref{mainresult}, the change of variables $((x_1,x_2,x_3),t)\mapsto ((u_1,u_2,u_3),s)$ defined by:
\begin{equation*}
\left\{\begin{array}{l}
u_1 =-\dfrac{x_1 +x_2 +x_3}{x_1^2 x_3^2}\\
u_2 =-\dfrac{x_2 (x_1 +x_2 +x_3)^2}{x_1^3 x_3^3}\\
u_3 =-\dfrac{(x_1 +x_2 +x_3)^2}{x_1^2 x_3^2}\\
ds=2(x_1-x_3)dt,
\end{array}\right.
\end{equation*}
for $(x_1,x_2,x_3)\in\Omega_{00}$ transforms the Lotka-Volterra system \eqref{LV} into the linear differential system:
\begin{equation*}
\left\{\begin{array}{l}
\dfrac{du_{1}}{ds}=u_1\\
\dfrac{du_{2}}{ds}=u_2\\
\dfrac{du_{3}}{ds}=u_3.\\
\end{array}\right.
\end{equation*}

\bigskip
\bigskip
Let us now analyze the Euler equations from the free rigid body dynamics (see e.g. \cite{holm1}, \cite{holm2}, \cite{holm3}, \cite{marsdenratiu}, \cite{ratiurazvan}). Recall that the Euler equations from the free rigid body dynamics are given by the differential system:
\begin{equation}\label{Ee}
\left\{\begin{array}{l}
\dfrac{dx_{1}}{dt}=\dfrac{I_2 -I_3}{I_2 I_3} x_2 x_3\\
\dfrac{dx_{2}}{dt}=\dfrac{I_3 -I_1}{I_1 I_3} x_1 x_3\\
\dfrac{dx_{3}}{dt}=\dfrac{I_1 -I_2}{I_1 I_2} x_1 x_2,\\
\end{array}\right.
\end{equation}
where the nonzero real numbers $I_1,I_2,I_3$ are the components of the inertia tensor. In the following we consider the case when $I_2 \neq I_3$.

If one denote: $$X=(\dfrac{I_2 -I_3}{I_2 I_3} x_2 x_3)\partial_{x_1}+(\dfrac{I_3 -I_1}{I_1 I_3} x_1 x_3)\partial_{x_2}+(\dfrac{I_1 -I_2}{I_1 I_2} x_1 x_2)\partial_{x_3},$$ then $\operatorname{div}(X)=0$, and hence the assumptions of the Theorem \eqref{mainresult} do not hold. 

The system \eqref{Ee} admits a Hamilton-Poisson realization of type \eqref{sysham}, where:
\begin{align*}
&\nu(x_1,x_2,x_3)=:\nu_{cst.}(x_1,x_2,x_3)=-1,\\
&C(x_1,x_2,x_3)=\dfrac{1}{2}(x_1^2 +x_2^2 +x_3^2),\\
&H(x_1,x_2,x_3)=\dfrac{1}{2}(\dfrac{x_1^2}{I_1} +\dfrac{x_2^2}{I_2} +\dfrac{x_3^2}{I_3}).
\end{align*}

In order to correct the vector field $X$ such that one can apply the Theorem \eqref{mainresult}, we use the Remark \eqref{okrem} and consequently  choose a rescaling function $\mu(x_1,x_2,x_3)=x_1$, and the associated new time transformation $dt=\mu(x_1,x_2,x_3) dt^{\prime}$ defined on the open and dense subset of $\R^3$ given by $\{(x_1,x_2,x_3)\in\R^3 \mid x_1 \neq 0\}$.

This new time transformation, transforms the system \eqref{Ee} into the differential system:
\begin{equation}\label{EC}
\left\{\begin{array}{l}
\dfrac{dx_{1}}{d t^\prime}=\dfrac{I_2 -I_3}{I_2 I_3} x_1 x_2 x_3\\
\dfrac{dx_{2}}{d t^\prime}=\dfrac{I_3 -I_1}{I_1 I_3} x_1^2 x_3\\
\dfrac{dx_{3}}{d t^\prime}=\dfrac{I_1 -I_2}{I_1 I_2} x_1^2 x_2,\\
\end{array}\right.
\end{equation}
Recall that the vector field which generates the differential system \eqref{EC} is $\mu\cdot X$:
$$\mu\cdot X=(\dfrac{I_2 -I_3}{I_2 I_3} x_1 x_2 x_3)\partial_{x_1}+(\dfrac{I_3 -I_1}{I_1 I_3} x_1^2 x_3)\partial_{x_2}+(\dfrac{I_1 -I_2}{I_1 I_2} x_1^2 x_2)\partial_{x_3}.$$
One note that the divergence of the vector field $\mu\cdot X$ is given by: 
$$
\operatorname{div}(\mu\cdot X)=\dfrac{I_2-I_3}{I_2 I_3}x_2 x_3.
$$
For $I_2 \neq I_3$ we have that $\operatorname{div}(\mu\cdot X)$ it is not identically zero.

The system \eqref{EC} admits a Hamilton-Poisson realization of the type \eqref{sysham}, where:
\begin{align*}
&\nu(x_1,x_2,x_3)=\nu_{cst.}(x_1,x_2,x_3)\cdot \mu(x_1,x_2,x_3)=-x_1,\\
&C(x_1,x_2,x_3)=\dfrac{1}{2}(x_1^2 +x_2^2 +x_3^2),\\
&H(x_1,x_2,x_3)=\dfrac{1}{2}(\dfrac{x_1^2}{I_1} +\dfrac{x_2^2}{I_2} +\dfrac{x_3^2}{I_3}).
\end{align*}
The sets introduced in Theorem \eqref{mainresult} in the case of the system \eqref{EC} are given by:
\begin{align*}
&\Omega=\R^3,\\
&\Omega_{0}=\{(x_1,x_2,x_3)\in\R^3 \mid x_1 \neq 0\},\\
&\mathcal{O}=\{(x_1,x_2,x_3)\in\R^3 \mid x_1 \neq 0;\ x_2 x_3=0\},\\
&\Omega_{00}=\{(x_1,x_2,x_3)\in\R^3 \mid x_1 x_2 x_3\neq 0\}.
\end{align*}
Then by Theorem \eqref{mainresult}, the change of variables $((x_1,x_2,x_3),t^\prime)\mapsto ((u_1,u_2,u_3),s)$ defined by:
\begin{equation*}
\left\{\begin{array}{l}
u_1 =-\dfrac{1}{x_1}\\
u_2 =-\dfrac{x_1^2 +x_2^2 +x_3^2}{2 x_1}\\
u_3 =-\dfrac{x_1}{2 I_1} -\dfrac{x_2^2}{2 x_1 I_2} -\dfrac{x_3^2}{2 x_1 I_3}\\
ds=\dfrac{I_3-I_2}{I_2 I_3}x_2 x_3 dt^\prime,
\end{array}\right.
\end{equation*}
for $(x_1,x_2,x_3)\in\Omega_{00}$ transforms the system \eqref{EC} into the linear differential system:
\begin{equation*}
\left\{\begin{array}{l}
\dfrac{du_{1}}{ds}=u_1\\
\dfrac{du_{2}}{ds}=u_2\\
\dfrac{du_{3}}{ds}=u_3.\\
\end{array}\right.
\end{equation*}

\subsection*{Acknowledgment}
I would like to thank Professor Metin G\"urses for indicating the references \cite{g1}, \cite{g2}.

\end{document}